\documentclass[12pt]{amsart}

\usepackage{tikz}

\theoremstyle{plain}
\newtheorem{thm}{Theorem}[section]

\newtheorem{cor}[thm]{Corollary}
\newtheorem{obs}[thm]{Observation}
\newtheorem{prop}[thm]{Proposition}

\theoremstyle{definition}

\DeclareMathOperator{\cyc}{cyc}

\newcommand{\ml}{\mathcal{L}}
\newcommand{\bp}{\mathcal{B}}
\newcommand{\mt}{\mathcal{T}}

\newcommand{\qbinom}[2]{\genfrac{[}{]}{0pt}{}{#1}{#2}}
\newcommand{\lbinom}[2]{\genfrac{\{}{\}}{0pt}{}{#1}{#2}}

\title{Unimodality via alternating gamma vectors}

\author[C. Brittenham]{Charles Brittenham}
\author[A. Carroll]{Andrew Carroll}
\author[T. K. Petersen]{T. Kyle Petersen}
\author[C. Thomas]{Connor Thomas}
\address{Department of Mathematical Sciences, DePaul University, Chicago, IL, USA}
\email{tpeter21@depaul.edu}
\urladdr{http://math.depaul.edu/tpeter21/}

\thanks{Work of Carroll partially supported by a faculty summer research grant from the College of Science and Health at DePaul University. Work of Petersen partially supported by a Simons Foundation collaboration grant. Work of Thomas supported by an undergraduate research assistantship from the College of Science and Health at DePaul University.}

\begin{document}

\begin{abstract}
For a polynomial with palindromic coefficients, unimodality is equivalent to having a nonnegative $g$-vector. A sufficient condition for unimodality is having a nonnegative $\gamma$-vector, though one can have negative entries in the $\gamma$-vector and still have a nonnegative $g$-vector.

In this paper we provide combinatorial models for three families of $\gamma$-vectors that alternate in sign. In each case, the $\gamma$-vectors come from unimodal polynomials with straightforward combinatorial descriptions, but for which there is no straightforward combinatorial proof of unimodality. 

By using the transformation from $\gamma$-vector to $g$-vector, we express the entries of the $g$-vector combinatorially, but as an alternating sum. In the case of the $q$-analogue of $n!$, we use a sign-reversing involution to interpret the alternating sum, resulting in a manifestly positive formula for the $g$-vector. In other words, we give a combinatorial proof of unimodality. We consider this a ``proof of concept" result that we hope can inspire a similar result for the other two cases, $\prod_{j=1}^n (1+q^j)$ and the $q$-binomial coefficient $\qbinom{n}{k}$.
\end{abstract}

\maketitle

\section{Introduction}

A sequence of numbers $a_1, a_2, \ldots$ is \emph{unimodal} if it never increases after the first time it decreases, i.e., if for some index $k$ we have $a_1 \leq \cdots \leq a_{k-1} \leq a_k \geq a_{k+1} \geq \cdots$. Unimodality problems abound in algebraic, enumerative, and topological combinatorics.  Many of the interesting examples involve Hilbert series of certain graded algebras, and combinatorial invariants of polytopes related to both face enumeration and lattice-point enumeration. Surveys on the topic include one by Stanley \cite{Stanleyunimodal}, another by Brenti \cite{Brentiunimodal}, and more recently, one from Br\"and\'en \cite{Brandensurvey}.

Unimodality can be surpisingly difficult to prove combinatorially, even when there is a good combinatorial understanding of the sequence. See, e.g., Zeilberger's discussion in \cite{Z}. Some stronger properties that imply unimodality include \emph{log-concavity}, \emph{real-rootedness} (of the corresponding generating function), and, under the assumption that the sequence is palindromic, \emph{gamma-nonnegativity}.

The main purpose of this paper is to show that there are interesting families of unimodal sequences whose gamma vectors (defined in Section \ref{sec:bases}) are not nonnegative, but in fact alternate in sign. Moreover, we give a combinatorial paradigm for how unimodality can be deduced from such gamma vectors. Specifically, we will discuss three families of polynomials: 
\begin{itemize}
\item the $q$-analogue of $n!$, $[n]! = \prod_{i=1}^n (1-q^i)/(1-q)$,
\item the $q$-binomial coefficients, $\qbinom{n}{k} = \frac{[n]!}{[k]![n-k]!}$, and 
\item the polynomials $\prod_{j=1}^n (1+q^j)$. 
\end{itemize}

For the first example, $[n]!$, unimodality follows easily from a lemma that says products of unimodal and palindromic polynomials are unimodal \cite{Andrews1}. It is well known that $[n]!$ is the generating function for permutations according to the number of inversions, so a combinatorial explanation of unimodality can be given with a family of injections that take permutations with $k-1$ inversions to permutations with $k$ inversions. Such maps are implicit in the fact that there is a symmetric chain decomposition of the Bruhat order on the symmetric group \cite[Section 7]{Stanley1980}. 

For the $q$-binomial coefficients, we have that $\qbinom{n}{k}$ counts lattice paths in a $k\times (n-k)$ box according to area below the path. Despite this simple interpretation, a combinatorial proof of unimodality was elusive for a long time. Stanley showed how unimodality follows from algebraic considerations \cite[Theorem 11]{Stanleyunimodal}, and a combinatorial proof was given by O'Hara in the 1990s, in what was a considered a combinatorial tour-de-force \cite{ohara}. See \cite{Z} for an exposition. 

The third kind of polynomial we discuss is given by $\prod_{j=1}^n (1+q^j)$. These polynomials count integer partitions with distinct parts of size at most $n$. Proof of unimodality follows from work of Odlyzko and Richardson \cite{OR} using analytic techniques. Stanley gave an algebraic explanation for unimodality and highlighted the absence of a combinatorial proof of unimodality in \cite[Example 3]{Stanleyunimodal}. To date, there is still no combinatorial proof of unimodality. See \cite{StanleyZanello} for recent related results.

In this paper, we provide a roadmap for proving unimodality for all these polynomials (Section \ref{sec:indirect}). The final step in the process is to find a sign-reversing involution on a set of decorated ballot paths, where the decorations depend on the family of polynomials under consideration. Sadly, we have only been clever enough to identify the sign-reversing involution in the simplest case, of $[n]!$. We provide details of the other two cases in the hope that others can use this idea to give new, combinatorial proofs of unimodality.

\section{Background and terminology}

\subsection{Two bases for palindromic polynomials}\label{sec:bases}

Suppose $h = \sum h_i q^i$ is a polynomial with integer coefficients
such that $q^dh(1/q) = h(q)$ for some positive integer $d$.  Then $h_i
= h_{d-i}$ and we will call such polynomial \emph{palindromic}. It is
easy to verify that if such a $d$ exists, it is the sum of largest and
smallest powers of $q$ appearing in $h$. We define this integer $d$ to
be the \emph{palindromic degree} of $h$. 

Palindromic polynomials of palindromic degree $d$ span a vector space of dimension roughly $d/2$, and with this in mind, we can express such polynomials in the bases
\[
 G_d = \{ q^i + \cdots + q^{d-i}\}_{0\leq 2i \leq d}
\]
and
\[
 \Gamma_d = \{ q^i (1+q)^{d-2i} \}_{0\leq 2i \leq d}.
\]
That is, if $h$ is palindromic there are vectors $g = (g_0, g_1, g_2, \ldots, g_{\lfloor d/2 \rfloor})$ and $\gamma = (\gamma_0, \gamma_1,\ldots, \gamma_{\lfloor d/2 \rfloor})$ such that
\begin{align}
 h(q) &= \sum_{0\leq 2i \leq d} g_i \sum_{j=i}^{d-i} q^j,\\
  &= \sum_{0 \leq 2i \leq d } \gamma_i q^i(1+q)^{d-2i}.
\end{align}
The vectors of coefficients are known as the \emph{$g$-vector} and
\emph{$\gamma$-vector} of $h$, respectively. 

It will be convenient to define the $\gamma$-polynomial of $h$ by
$\gamma_h(z)=\sum_{i=0}^{\lfloor d/2\rfloor} \gamma_i z^i$ and
the $g$-polynomial of $h$ by $g_h(z)=\sum_{i=0}^{\lfloor
  d/2\rfloor} g_i z^i$.  The general study of $g$-vectors and $g$-polynomials dates to McMullen's conjectures of the early 1970s, see \cite{Stanleycomm}, while work on $\gamma$-polynomials dates at least to 1985 work of Andrews \cite{Andrews} (and appears notably in 1970 work of Foata and Sch\"utzenberger on Eulerian polynomials \cite{FS}). More recently Gal made connections between $\gamma$-vectors and the same sort of topological questions that inspired work with $g$-polynomials \cite{Gal}.
  
The definition of the $\gamma$-polynomial immediately yields the following observation about the multiplicative nature of $\gamma$-polynomials. 
    
%

\begin{obs}\label{obs:gamma-poly-factors}
  If $h$ and $k$ are palindromic polynomials, then $h\cdot
  k$ is palindromic and its $\gamma$-polynomial is $\gamma_{h\cdot
    k} = \gamma_h \cdot \gamma_k$.  
\end{obs} 

In particular, to calculate the $\gamma$-vectors of products of polynomials, it is enough to calculate the $\gamma$-vectors of their factors. Notice that the $g$-polynomial is \emph{not} multiplicative in the same way, e.g., the $g$-polynomial of $1+q$ is $g_{1+q}(z)=1$, while the $g$-polynomial of $(1+q)^2$ is $g_{(1+q)^2}(z) = 1+z \neq 1^2$.

\subsection{Ballot paths}

In \cite[Section 6]{NPT} Nevo, Petersen, and Tenner give the following transformation that provides a change of basis from $\Gamma_d$ to $G_d$:
\[
B_d:=\left[ \binom{d-2j}{i-j}-\binom{d-2j}{i-j-1}\right]_{0 \leq i,j \leq d/2}.
\]
These are lower triangular matrices with ones on the diagonal, so they are invertible and their inverses are also lower triangular with ones on the diagonal.

Let $B_d(i,j)$ denote the $(i,j)$ entry of this change of basis matrix, i.e.,
\[
 B_d(i,j) = \binom{d-2j}{i-j}-\binom{d-2j}{i-j-1}.
\]
A standard combinatorial interpretation for these numbers is the number of \emph{ballot paths} of length $d-2j$ with $i-j$ North steps. Recall ballot paths are lattice paths that start at $(0,0)$ and take steps ``East'' from $(i,j)$ to $(i+1,j)$ and ``North'' from $(i,j)$ to $(i,j+1)$, such that the path never crosses above the line $y=x$ in the cartesian plane. (Ballot paths that end on the diagonal at $(n,n)$ are known as \emph{Dyck paths}, and are counted by Catalan numbers.) 

In terms of words, we can encode ballot paths as words on the alphabet
$\{N, E\}$ such that no initial segment of the word contains more
letters $N$ than letters $E$. The  $\binom{6}{2}-\binom{6}{1}=9$
ballot paths of length six with two North steps are displayed in
Figure \ref{fig:ballot-paths}.

\begin{figure}
\[
 \begin{tikzpicture}[scale=.5]
 \draw[dotted] (0,0) grid (4,4);
 \draw[dashed] (-.5,-.5)--(4.5,4.5);
  \draw (0,0) node[circle, inner sep=2, fill=black] {} -- (1,0) node[circle, inner sep=2, fill=black] {} -- (1,1) node[circle, inner sep=2, fill=black] {} -- (2,1) node[circle, inner sep=2, fill=black] {}-- (2,2) node[circle, inner sep=2, fill=black] {}-- (3,2) node[circle, inner sep=2, fill=black] {} -- (4,2) node[circle, inner sep=2, fill=black] {};
  \draw (2,-1) node {$ENENEE$};
 \end{tikzpicture}
 \quad
 \begin{tikzpicture}[scale=.5]
 \draw[dotted] (0,0) grid (4,4);
 \draw[dashed] (-.5,-.5)--(4.5,4.5);
  \draw (0,0) node[circle, inner sep=2, fill=black] {} -- (1,0) node[circle, inner sep=2, fill=black] {} -- (1,1) node[circle, inner sep=2, fill=black] {} -- (2,1) node[circle, inner sep=2, fill=black] {}-- (3,1) node[circle, inner sep=2, fill=black] {}-- (3,2) node[circle, inner sep=2, fill=black] {} -- (4,2) node[circle, inner sep=2, fill=black] {};
  \draw (2,-1) node {$ENEENE$};
 \end{tikzpicture}
 \quad
 \begin{tikzpicture}[scale=.5]
 \draw[dotted] (0,0) grid (4,4);
 \draw[dashed] (-.5,-.5)--(4.5,4.5);
  \draw (0,0) node[circle, inner sep=2, fill=black] {} -- (1,0) node[circle, inner sep=2, fill=black] {} -- (1,1) node[circle, inner sep=2, fill=black] {} -- (2,1) node[circle, inner sep=2, fill=black] {}-- (3,1) node[circle, inner sep=2, fill=black] {}-- (4,1) node[circle, inner sep=2, fill=black] {} -- (4,2) node[circle, inner sep=2, fill=black] {};
  \draw (2,-1) node {$ENEEEN$};
 \end{tikzpicture}
\]
\[
 \begin{tikzpicture}[scale=.5]
 \draw[dotted] (0,0) grid (4,4);
 \draw[dashed] (-.5,-.5)--(4.5,4.5);
  \draw (0,0) node[circle, inner sep=2, fill=black] {} -- (1,0) node[circle, inner sep=2, fill=black] {} -- (2,0) node[circle, inner sep=2, fill=black] {} -- (2,1) node[circle, inner sep=2, fill=black] {}-- (2,2) node[circle, inner sep=2, fill=black] {}-- (3,2) node[circle, inner sep=2, fill=black] {} -- (4,2) node[circle, inner sep=2, fill=black] {};
  \draw (2,-1) node {$EENNEE$};
 \end{tikzpicture}
 \quad
 \begin{tikzpicture}[scale=.5]
 \draw[dotted] (0,0) grid (4,4);
 \draw[dashed] (-.5,-.5)--(4.5,4.5);
  \draw (0,0) node[circle, inner sep=2, fill=black] {} -- (1,0) node[circle, inner sep=2, fill=black] {} -- (2,0) node[circle, inner sep=2, fill=black] {} -- (2,1) node[circle, inner sep=2, fill=black] {}-- (3,1) node[circle, inner sep=2, fill=black] {}-- (3,2) node[circle, inner sep=2, fill=black] {} -- (4,2) node[circle, inner sep=2, fill=black] {};
  \draw (2,-1) node {$EENENE$};
 \end{tikzpicture}
 \quad
 \begin{tikzpicture}[scale=.5]
 \draw[dotted] (0,0) grid (4,4);
 \draw[dashed] (-.5,-.5)--(4.5,4.5);
  \draw (0,0) node[circle, inner sep=2, fill=black] {} -- (1,0) node[circle, inner sep=2, fill=black] {} -- (2,0) node[circle, inner sep=2, fill=black] {} -- (2,1) node[circle, inner sep=2, fill=black] {}-- (3,1) node[circle, inner sep=2, fill=black] {}-- (4,1) node[circle, inner sep=2, fill=black] {} -- (4,2) node[circle, inner sep=2, fill=black] {};
  \draw (2,-1) node {$EENEEN$};
 \end{tikzpicture}
\]
\[
 \begin{tikzpicture}[scale=.5]
 \draw[dotted] (0,0) grid (4,4);
 \draw[dashed] (-.5,-.5)--(4.5,4.5);
  \draw (0,0) node[circle, inner sep=2, fill=black] {} -- (1,0) node[circle, inner sep=2, fill=black] {} -- (2,0) node[circle, inner sep=2, fill=black] {} -- (3,0) node[circle, inner sep=2, fill=black] {}-- (3,1) node[circle, inner sep=2, fill=black] {}-- (3,2) node[circle, inner sep=2, fill=black] {} -- (4,2) node[circle, inner sep=2, fill=black] {};
  \draw (2,-1) node {$EEENNE$};
 \end{tikzpicture}
 \quad
 \begin{tikzpicture}[scale=.5]
 \draw[dotted] (0,0) grid (4,4);
 \draw[dashed] (-.5,-.5)--(4.5,4.5);
  \draw (0,0) node[circle, inner sep=2, fill=black] {} -- (1,0) node[circle, inner sep=2, fill=black] {} -- (2,0) node[circle, inner sep=2, fill=black] {} -- (3,0) node[circle, inner sep=2, fill=black] {}-- (3,1) node[circle, inner sep=2, fill=black] {}-- (4,1) node[circle, inner sep=2, fill=black] {} -- (4,2) node[circle, inner sep=2, fill=black] {};
  \draw (2,-1) node {$EEENEN$};
 \end{tikzpicture}
 \quad
 \begin{tikzpicture}[scale=.5]
 \draw[dotted] (0,0) grid (4,4);
 \draw[dashed] (-.5,-.5)--(4.5,4.5);
  \draw (0,0) node[circle, inner sep=2, fill=black] {} -- (1,0) node[circle, inner sep=2, fill=black] {} -- (2,0) node[circle, inner sep=2, fill=black] {} -- (3,0) node[circle, inner sep=2, fill=black] {}-- (4,0) node[circle, inner sep=2, fill=black] {}-- (4,1) node[circle, inner sep=2, fill=black] {} -- (4,2) node[circle, inner sep=2, fill=black] {};
  \draw (2,-1) node {$EEEENN$};
 \end{tikzpicture}
\]
\caption{Ballot paths of length 6 with 2 North steps.}\label{fig:ballot-paths}
\end{figure}
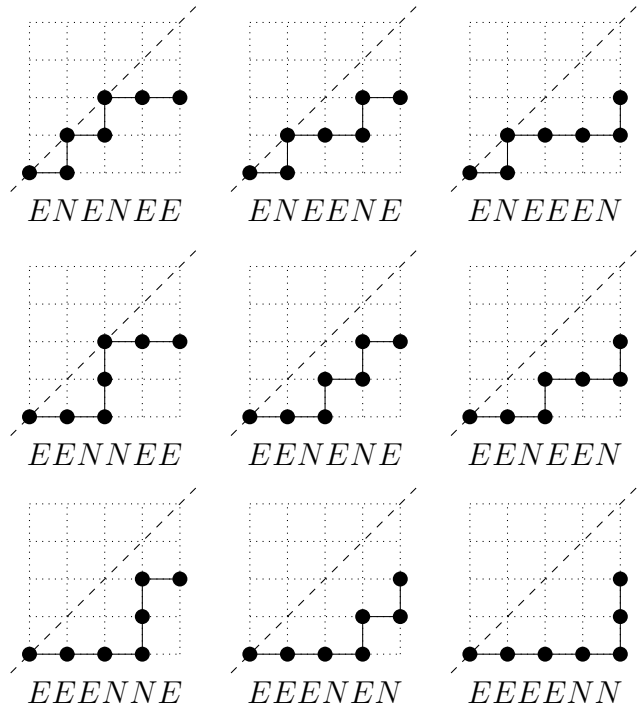

Let $\bp_d(i,j)$ denote the set of paths of length $d-2j$ with $i-j$ North steps, so that $|\bp_d(i,j)| = B_d(i,j)$. We will have more to say about ballot paths in Section \ref{sec:gamma-to-g}.

\subsection{Unimodality}\label{sec:unimodal}

A polynomial $h(q) = \sum h_i q^i$ is called \emph{unimodal} if there is an index $k$ such that
\[
h_0 \leq \cdots \leq h_{k-1} \leq h_k \geq h_{k+1} \geq \cdots,
\] 
When $h$ is palindromic, unimodality is equivalent to saying that the $g$-vector is nonnegative, i.e.,
\[
 g_i = h_i-h_{i-1} \geq 0 ~\textrm{for $i\leq d/2$},
\]
where $d$ is the palindromic degree of $h$.

The elements of the basis $\Gamma_d$ are obviously unimodal. Thus a sufficient condition for unimodality of $h$ is that its $\gamma$-vector is nonnegative. However, this condition is far from necessary. 

For example, consider
\[
 h(q) = 1 + 3q+5q^2 + 6q^3 + 5q^4 + 3q^5 + q^6.
\]
This polynomial is clearly palindromic and unimodal. Its $g$-vector is $g = (1, 2, 2, 1)$, since
\begin{align*}
 h(q) &= 1\cdot (1+q+q^2 +q^3+q^4+q^5+q^6) \\
     & +2\cdot (q+q^2 +q^3+q^4+q^5)\\
     & +2\cdot (q^2 +q^3+q^4)\\
     & +1 \cdot (q^3),
\end{align*}
while its $\gamma$-vector is $\gamma= (1, -3,2,0)$, since
\begin{align*}
h(q) &= 1\cdot (1+6q+15q^2 +20q^3+15q^4+6q^5+q^6) \\
     & -3\cdot (q+4q^2 +6q^3+4q^4+q^5)\\
     & +2\cdot (q^2 +2q^3+q^4)\\
     & +0 \cdot (q^3).
\end{align*}

\subsection{A paradigm for proving unimodality}\label{sec:indirect}

The obvious approach to proving unimodality for a palindromic polynomial is to consider the entries of the $g$-vector directly. In a combinatorial setting, we want to know: what do the entries of the $g$-vector count?

In general, this is not as easy as it appears, even for simple examples. See Zeilberger's article on this topic in combinatorial enumeration \cite{Z}.

Our main purpose with this article is to give a new paradigm for combinatorial proofs of unimodality. We consider examples in which the combinatorial description of the $g$-vector is difficult or nonexistent, but for which the $\gamma$-vector is not too bad. Then we use the change of basis matrix $B_d$ to write
\begin{equation}\label{eq:gamtog}
 g_i = \sum_{j\geq 0} \gamma_j B_d(i,j).
\end{equation}

Now, if the $\gamma$-vector contained only nonnegative entries, unimodality follows immediately, so we focus on ``interesting'' examples, in which the entries of the $\gamma$-vector alternate between positive and negative.  This means that Equation \eqref{eq:gamtog} is an alternating sum formula. Such formulas can be found throughout mathematics, and in combinatorics, the way one understands such a formula is via a ``sign-reversing involution." See Benjamin and Quinn's excellent introduction to the subject \cite{BQ}.

For example, with $h$ as in the previous section, we have that $h$ is palindromic of degree $d=6$ and $\gamma=(1,-3,2,0)$. The transformation $B_6$ is
\[
 B_6 = \left[ \begin{array}{cccc} 1 & 0 & 0 & 0 \\
  5 & 1 & 0 & 0 \\
  9 & 3 & 1 & 0 \\
  5 & 2 & 1 & 1
  \end{array} \right].
\]
Computing $B_6 \gamma$ we get
\[
 B_6 \gamma = 1\cdot \left[ \begin{array}{c} 1 \\ 5 \\ 9 \\ 5 \end{array}\right] - 3\cdot \left[ \begin{array}{c} 0 \\ 1 \\ 3 \\ 2 \end{array}\right] + 2\cdot \left[ \begin{array}{c} 0 \\ 0 \\ 1 \\ 1 \end{array}\right] + 0 \cdot \left[ \begin{array}{c} 0 \\ 0 \\ 0 \\ 1 \end{array}\right] = \left[ \begin{array}{c} 1 \\ 2 \\ 2 \\ 1 \end{array}\right] = g.
\]

That is,
\begin{align*}
g_0 &= 1\cdot 1 = 1,\\
g_1 &= 1\cdot 5 - 3 \cdot 1 =2,\\
g_2 &= 1\cdot 9 - 3\cdot 3 + 2\cdot 1 =2,\\
 g_3 &= 1 \cdot 5 - 3\cdot 2 + 2 \cdot 1 + 0 \cdot 1 =1.
\end{align*}

In Section \ref{sec:gamma-to-g}, we will see that $\gamma_j B_d(i,j)$ counts pairs of the form (matching, path) and that the parity of $j$ determines the sign of the pair. We will describe a matching on these pairs that puts together pairs of opposite sign. Then $g_i$ counts the pairs that have no partner in the matching. If we construct our matching optimally, all the leftovers will have the same (positive) sign.   

\section{Three families of alternating gamma vectors}
\label{sec:gamma-vectors-three}

In this section, we will present three families of palindromic polynomials whose gamma vectors alternate in sign. The particular $\gamma_j$ we come up with are described in terms of interesting and well-studied polynomials: the \emph{Fibonacci polynomials}, the \emph{Lucas polynomials}, and the \emph{lucanomial (fibinomial) coefficients}. Our discussion of these relies heavily on facts established in work of Sagan and Savage \cite{Sagan-Savage} and in Amdeberhan, Chen, Moll, and Sagan \cite{ACMS}. In each case the combinatorial model for the gamma vector is a set of matchings counted according to the number of edges.

\subsection{Fibonacci polynomials and the $q$-factorials}
\label{sec:q-analogue-n}

Let 
\[
F_n(s,t) = \sum_{T \in 1\times n} s^{ m(T)}t^{d(T)},
\]
where the sum is over all monomer-dimer tilings of a $1\times n$ rectangle, or equivalently, the number of (not necessarily perfect) matchings of a path of length $n-1$. The statistic $m(T)$ is the number of monomers (isolated nodes) in $T$ and $d(T)$ is the number of dimers (edges) in $T$. The \emph{weight} of a matching is $w(T)=s^{m(T)}t^{d(T)}$. 

For example with $n=6$ there are thirteen matchings, shown in Figure \ref{fig:matching}. We find the Fibonacci polynomial $F_6(s,t)$ is
\[
 F_6(s,t) = s^6 + 5s^4t + 6s^2t^2 + t^3.
\]

\begin{figure}
\[
\begin{array}{c | c}
\mbox{matching} & \mbox{weight} \\
\hline
& \\
\begin{tikzpicture}
\foreach \x in {0,...,5}{
\draw (\x,0) node[circle, inner sep=2, fill=black] {};
}
\draw (0,0) -- (1,0);
\draw (2,0) -- (3,0);
\draw (4,0) -- (5,0);
\end{tikzpicture}
&
t^3 \\
& \\
\begin{tikzpicture}
\foreach \x in {0,...,5}{
\draw (\x,0) node[circle, inner sep=2, fill=black] {};
}
\draw (2,0) -- (3,0);
\draw (4,0) -- (5,0);
\end{tikzpicture}
&
s^2t^2 \\
& \\
\begin{tikzpicture}
\foreach \x in {0,...,5}{
\draw (\x,0) node[circle, inner sep=2, fill=black] {};
}
\draw (1,0) -- (2,0);
\draw (4,0) -- (5,0);
\end{tikzpicture}
&
s^2t^2 \\
& \\
\begin{tikzpicture}
\foreach \x in {0,...,5}{
\draw (\x,0) node[circle, inner sep=2, fill=black] {};
}
\draw (0,0) -- (1,0);
\draw (4,0) -- (5,0);
\end{tikzpicture}
&
s^2t^2 \\
& \\
\begin{tikzpicture}
\foreach \x in {0,...,5}{
\draw (\x,0) node[circle, inner sep=2, fill=black] {};
}
\draw (4,0) -- (5,0);
\end{tikzpicture}
&
s^4t \\
& \\
\begin{tikzpicture}
\foreach \x in {0,...,5}{
\draw (\x,0) node[circle, inner sep=2, fill=black] {};
}
\draw (1,0) -- (2,0);
\draw (3,0) -- (4,0);
\end{tikzpicture}
&
s^2t^2 \\
& \\
\begin{tikzpicture}
\foreach \x in {0,...,5}{
\draw (\x,0) node[circle, inner sep=2, fill=black] {};
}
\draw (0,0) -- (1,0);
\draw (3,0) -- (4,0);
\end{tikzpicture}
&
s^2t^2 \\
& \\
\begin{tikzpicture}
\foreach \x in {0,...,5}{
\draw (\x,0) node[circle, inner sep=2, fill=black] {};
}
\draw (3,0) -- (4,0);
\end{tikzpicture}
&
s^4t \\
& \\
\begin{tikzpicture}
\foreach \x in {0,...,5}{
\draw (\x,0) node[circle, inner sep=2, fill=black] {};
}
\draw (0,0) -- (1,0);
\draw (2,0) -- (3,0);
\end{tikzpicture}
&
s^2t^2 \\
& \\
\begin{tikzpicture}
\foreach \x in {0,...,5}{
\draw (\x,0) node[circle, inner sep=2, fill=black] {};
}
\draw (2,0) -- (3,0);
\end{tikzpicture}
&
s^4t \\
& \\
\begin{tikzpicture}
\foreach \x in {0,...,5}{
\draw (\x,0) node[circle, inner sep=2, fill=black] {};
}
\draw (1,0) -- (2,0);
\end{tikzpicture}
&
s^4t \\
& \\
\begin{tikzpicture}
\foreach \x in {0,...,5}{
\draw (\x,0) node[circle, inner sep=2, fill=black] {};
}
\draw (0,0) -- (1,0);
\end{tikzpicture}
&
s^4t \\
& \\
\begin{tikzpicture}
\foreach \x in {0,...,5}{
\draw (\x,0) node[circle, inner sep=2, fill=black] {};
}
\end{tikzpicture}
&
s^6 \\
& \\
\end{array}
\]
\caption{The thirteen Fibonacci matchings for $n=6$, with corresponding weight.}\label{fig:matching}
\end{figure}

There are many formulas and recurrences for the Fibonacci polynomials that generalize known facts for Fibonacci numbers. Such results for these polynomials can be found, e.g., in \cite{ACMS, Cigler, Sagan-Savage}. Here we mention two straightforward observations.

First, since every matching must end in either an edge or an isolated node, we get
\begin{equation}\label{eq:fib-poly-recur}
 F_n(s,t) = sF_{n-1}(s,t) + tF_{n-2}(s,t),
\end{equation}
with $F_0(s,t)=1$ and  $F_1(s,t) = s$. This generalizes the numeric recurrence for Fibonacci numbers.

Further, the familiar identity for Fibonacci numbers as a sum of binomial coefficients, $F_n = \sum_{k \geq 0} \binom{n-k}{k}$, is generalized to
\begin{equation}\label{eq:binomexpansion}
 F_n(s,t) = \sum_{k\geq 0} \binom{n-k}{k} s^{n-2k}t^k.
\end{equation}
For example, 
\[
F_6(s,t) = \binom{6}{0}s^6 + \binom{5}{1}s^4t + \binom{4}{2}s^2t^2 + \binom{3}{3}t^3.
\]
This identity has a straightforward combinatorial interpretation: if a matching has $k$ edges (and hence $n-2k$ isolated nodes), then there are $(n-2k)+k = n-k$ items to be positioned, and we have to choose $k$ of the positions in which to place the edges. For convenience later on, we will denote by $T(n,k)$ the set of all matchings of the $1\times n$ path graph with $k$ edges, so that
\[
 |T(n,k)| = \binom{n-k}{k}.
\]

An important specialization of $F_n(s,t)$ is obtained by setting $s=1+q$ and $t=-q$. By induction and the recurrence in \eqref{eq:fib-poly-recur} we quickly get
\begin{equation}\label{eq:Ftoq}
 F_n(1+q,-q) = 1+q + q^2 + \cdots + q^n, 
\end{equation}
which is denoted by $[n+1]_q$ (or simply by $[n+1]$ when the $q$ is understood from the context), and called the \emph{$q$-analogue} of $n+1$. 


Note that $[n]$ has palindromic degree $n-1$, and hence it has a $\gamma$-vector as defined in Section \ref{sec:bases}. With Equation \eqref{eq:Ftoq} together with \eqref{eq:binomexpansion} we get 
\[
[n] = F_{n-1}(1+q,-q)= \sum_{k\geq 0} (-1)^k\binom{n-1-k}{k} q^k (1+q)^{n-1-2k},
\]
from which we conclude that the $\gamma$-vector of $[n]$ has 
\[
 \gamma_j = (-1)^j\binom{n-1-j}{j} = (-1)^j|T(n-1,j)|,
\]
i.e., up to sign it counts matchings on a path with $n-1$ nodes that have $j$ edges.

To state this another way, $F_{n-1}(1,-z)$ is the $\gamma$-polynomial for $[n]$. 

Now recall the $q$-analogue of $n!$ is defined to be
\[
[n]!=[n][n-1]\cdots [2][1]. 
\]
Since $\gamma$-polynomials are multiplicative (Observation \ref{obs:gamma-poly-factors}), we have established the following proposition.

\begin{prop}\label{prop:gam-poly-q-fact}
The $\gamma$-polynomial for $[n]$ is $F_{n-1}(1,-z)$ for any $n\geq 1$. Hence, the $\gamma$-polynomial for $[n]!$ is 
\[ \prod_{i=0}^{n-1} F_i(1,-z) = F_0(1,-z) F_1(1,-z) \cdots F_{n-1}(1,-z).\]
\end{prop}

Thus we can give the following combinatorial interpretation to the $\gamma$-vector for $[n]!$.

\begin{cor}
  If $(\gamma_0, \ldots, \gamma_{\lfloor d/2\rfloor})$ is the
  $\gamma$-vector of $[n]!$, then 
  \[
  \gamma_i=(-1)^i \left| \bigcup_{j_1+j_2+\cdots+j_{n-1}=i} T(1,j_1)\times T(2,j_2)\times \cdots \times T(n-1,j_{n-1}) \right|.
  \]  
\end{cor}

In other words, up to sign, the $i$th entry of the $\gamma$-vector for $[n]!$ counts $(n-1)$-tuples of matchings (of paths with $1,2,\ldots,n-1$ nodes) such that there are a total of $i$ edges among all the matchings. We will call such tuples of matchings \emph{fibotorial matchings}. Denote the set of all fibotorial matchings with $i$ edges by 
\[
\mt(n,i) =\bigcup_{j_1+j_2+\cdots+j_{n-1}=i} T(1,j_1)\times T(2,j_2)\times \cdots \times T(n-1,j_{n-1}),
\]
and let $\mt_n$ denote the set of all fibotorial matchings for fixed $n$:
\[
 \mt_n = \bigcup_{i\geq 0} \mt(n,i).
\]

Setting $z=-1$ in Proposition \ref{prop:gam-poly-q-fact}, we get the following fun corollary.

\begin{cor}
 Let $\gamma_{[n]!}(z)$ denote the $\gamma$-polynomial of the $q$-analogue of $n!$. Then $\gamma_{[n]!}(-1) = |\mt_n| = F_0F_1F_2\dotsc F_{n-1}$ is the
  product of the first $n$ Fibonacci numbers, with initial values
  $F_0=F_1=1$. 
\end{cor}

This curious result was proved differently by Doron Zeilberger. Our approach is essentially Richard Stanley's, and Johann Cigler has a similar argument. See Zeilberger's note \cite{Z2}.

\subsection{Lucas polynomials and partitions with distinct parts}
\label{sec:polynomials-1+qn}

The Lucas polynomials are defined analogously to the Fibonacci polynomials. We have \[
L_n(s,t) = \sum_{T \in \cyc_n} s^{m(T)} t^{d(T)},
\]
where the sum is over all matchings on a cycle graph with $n$ nodes and edges. For example, in Figure \ref{fig:lucas} we see the seven matchings of a $4$-cycle, from which we find $L_n(s,t) = s^4+4s^2t+2t^2$. 

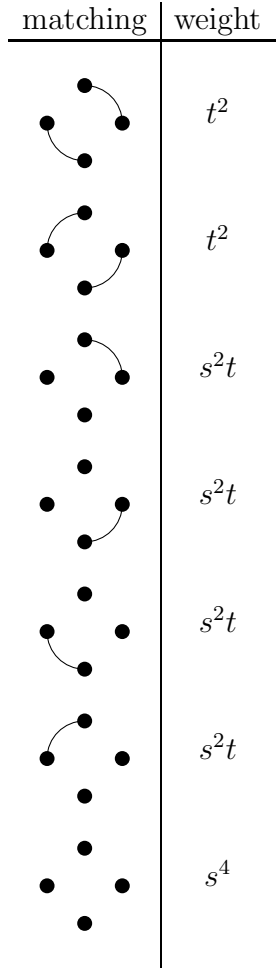
\begin{figure}
\[
\begin{array}{c | c}
\mbox{matching} & \mbox{weight} \\
\hline
& \\
\begin{tikzpicture}[scale=.5,baseline=0]
\draw (1,0) node[circle, inner sep=2, fill=black] {};
\draw (-1,0) node[circle, inner sep=2, fill=black] {};
\draw (0,1) node[circle, inner sep=2, fill=black] {};
\draw (0,-1) node[circle, inner sep=2, fill=black] {};
\draw (1,0) arc (0:90:1);
\draw (-1,0) arc (180:270:1);
\end{tikzpicture}
&
t^2 \\
& \\
\begin{tikzpicture}[scale=.5,baseline=0]
\draw (1,0) node[circle, inner sep=2, fill=black] {};
\draw (-1,0) node[circle, inner sep=2, fill=black] {};
\draw (0,1) node[circle, inner sep=2, fill=black] {};
\draw (0,-1) node[circle, inner sep=2, fill=black] {};
\draw (1,0) arc (0:-90:1);
\draw (-1,0) arc (180:90:1);
\end{tikzpicture}
&
t^2 \\
& \\
\begin{tikzpicture}[scale=.5,baseline=0]
\draw (1,0) node[circle, inner sep=2, fill=black] {};
\draw (-1,0) node[circle, inner sep=2, fill=black] {};
\draw (0,1) node[circle, inner sep=2, fill=black] {};
\draw (0,-1) node[circle, inner sep=2, fill=black] {};
\draw (1,0) arc (0:90:1);
\end{tikzpicture}
&
s^2t \\
& \\
\begin{tikzpicture}[scale=.5,baseline=0]
\draw (1,0) node[circle, inner sep=2, fill=black] {};
\draw (-1,0) node[circle, inner sep=2, fill=black] {};
\draw (0,1) node[circle, inner sep=2, fill=black] {};
\draw (0,-1) node[circle, inner sep=2, fill=black] {};
\draw (1,0) arc (0:-90:1);
\end{tikzpicture}
&
s^2t \\
& \\
\begin{tikzpicture}[scale=.5,baseline=0]
\draw (1,0) node[circle, inner sep=2, fill=black] {};
\draw (-1,0) node[circle, inner sep=2, fill=black] {};
\draw (0,1) node[circle, inner sep=2, fill=black] {};
\draw (0,-1) node[circle, inner sep=2, fill=black] {};
\draw (-1,0) arc (180:270:1);
\end{tikzpicture}
&
s^2t \\
& \\
\begin{tikzpicture}[scale=.5,baseline=0]
\draw (1,0) node[circle, inner sep=2, fill=black] {};
\draw (-1,0) node[circle, inner sep=2, fill=black] {};
\draw (0,1) node[circle, inner sep=2, fill=black] {};
\draw (0,-1) node[circle, inner sep=2, fill=black] {};
\draw (-1,0) arc (180:90:1);
\end{tikzpicture}
&
s^2t \\
& \\
\begin{tikzpicture}[scale=.5,baseline=0]
\draw (1,0) node[circle, inner sep=2, fill=black] {};
\draw (-1,0) node[circle, inner sep=2, fill=black] {};
\draw (0,1) node[circle, inner sep=2, fill=black] {};
\draw (0,-1) node[circle, inner sep=2, fill=black] {};
\end{tikzpicture}
&
s^4 \\
& \\
\end{array}
\]
\caption{The seven Lucas matchings for $n=4$, with corresponding weight.}\label{fig:lucas}
\end{figure}

The Lucas polynomials satisfy the same recurrence as the Fibonacci polynomials, i.e.,
\begin{equation}\label{eq:Lpoly}
 L_n(s,t) = sL_{n-1}(s,t) + tL_{n-2}(s,t),
\end{equation}
but with initial conditions $L_0(s,t) = 2$ and $L_1(s,t) = s$. In terms of the tilings, this recurrence is easiest to understand by first relating $L_n(s,t)$ with $F_n(s,t)$. We have
\begin{equation}\label{eq:LF}
 L_n(s,t) = sF_{n-1}(s,t) + 2tF_{n-2}(s,t).
\end{equation}
This identity follows by considering the neighborhood of a fixed node. This node is either isolated, or it is connected to one of the nodes on either side of it. If the node is isolated, the remaining $n-1$ nodes form a linear matching whose weight is counted by $F_{n-1}(s,t)$. If the node is matched to an adjacent node, the remaining $n-2$ nodes form a linear matching, with weight $F_{n-2}(s,t)$.

Equation \eqref{eq:Lpoly} follows by using \eqref{eq:LF} and then applying the Fibonacci recurrence \eqref{eq:fib-poly-recur} as follows: 
\begin{align*}
sL_n(s,t) + tL_{n-1}(s,t) &= s^2F_{n-1}(s,t) + 2stF_{n-2}(s,t) + stF_{n-2}(s,t) + 2t^2F_{n-3}(s,t),\\
 &= s(sF_{n-1}(s,t) + tF_{n-2}(s,t)) + 2t(sF_{n-2}(s,t) + tF_{n-3}(s,t)),\\
 &= sF_n(s,t) + 2tF_{n-1}(s,t),\\
 &= L_{n+1}(s,t).
\end{align*}

By induction we also have the following explicit formula for Lucas polynomials:
\begin{equation}\label{eq:Lexplicit}
 L_n(s,t) = \sum_{k\geq 0} \frac{n}{n-k}\binom{n-k}{k} s^{n-2k}t^k.
\end{equation}
Letting $T'(n,k)$ denote the number of matchings of an $n$-cycle with $k$ edges, we have
\[
 |T'(n,k)|=\frac{n}{n-k}\binom{n-k}{k}.
\]

Finally, note that we have the following specialization of the Lucas polynomials. If we set $s=1+q$ and $t=-q$, then using induction and \eqref{eq:Lpoly} we get
\begin{equation}\label{eq:1+qn}
L_n(1+q,-q) = 1+q^n.
\end{equation}

Since $1+q^n$ has palindromic degree $n$, it has a $\gamma$-vector. Putting \eqref{eq:Lexplicit} together with \eqref{eq:1+qn} we have
\[
 1+q^n = L_n(1+q,-q) = \sum_{k\geq 0} (-1)^k\frac{n}{n-k}\binom{n-k}{k} q^k (1+q)^{n-2k},
\]
from which we conclude that the $\gamma$-vector of $1+q^n$ has
\[
 \gamma_j = (-1)^j\frac{n}{n-j}\binom{n-j}{j} = (-1)^j|T'(n,j)|,
\]
so that up to sign, the $\gamma$-vector counts matchings on an $n$-cycle with $j$ edges.

Stated differently, $L_n(1,-z)$ is the $\gamma$-polynomial for $1+q^n$.

Now fix $n$ and consider the polynomial
\[
 \prod_{j=1}^n (1+q^j).
\]
The coefficient of $q^i$ in this polynomial is equal to the number of integer partitions of $i$ with at most $n$ nonzero parts, all distinct, and each bounded by $n$. This polynomial is discussed at length in \cite{Stanleyunimodal}. While it is known to be unimodal, to date there is no combinatorial proof of its unimodality. See also \cite{StanleyZanello}.

By the multiplicative nature of $\gamma$-polynomials, we get the following result, which is similar to Proposition \ref{prop:gam-poly-q-fact}.

\begin{prop}\label{prop:gam-partn}
The $\gamma$-polynomial for $1+q^n$ is $L_n(1,-z)$ for any $n\geq 1$. Hence, the $\gamma$-polynomial for $\prod_{j=1}^n (1+q^j)$ is
\[
 \prod_{i=j}^n L_j(1,-z) = L_1(1,-z) L_2(1,-z)\cdots L_n(1,-z).
\]
\end{prop}

Therefore we have the following combinatorial interpretation for the $\gamma$-vector of $\prod_{j=1}^n (1+q^j)$.

\begin{cor}
  If $(\gamma_0, \ldots, \gamma_{\lfloor d/2\rfloor})$ is the
  $\gamma$-vector of $\prod_{j=1}^n (1+q^j)$, then 
  \[
  \gamma_i=(-1)^i \left| \bigcup_{j_1+j_2+\cdots+j_n=i} T'(1,j_1)\times T'(2,j_2)\times \cdots \times T'(n,j_n) \right|.
  \]  
\end{cor}

In other words, up to sign, the $i$th entry of the $\gamma$-vector for $\prod_{j=1}^n (1+q^j)$ counts $n$-tuples of matchings (of cycles with $1,2,\ldots,n$ nodes) such that there are a total of $i$ edges among all the matchings. We will call such a tuple of matchings a \emph{lucatorial matching}. Denote the set of lucatorial matchings with $i$ edges by
\[
 \mt'(n,i) = \bigcup_{j_1+j_2+\cdots+j_n=i} T'(1,j_1)\times T'(2,j_2)\times \cdots \times T'(n,j_n),
\] 
and denote the set of all lucatorial matchings with fixed $n$ by
\[
 \mt'_n = \bigcup_{i\geq 0} \mt'(n,i).
\]

\subsection{Lucanomial coefficients and $q$-binomial coefficients}
\label{sec:q-binom-coeff}

The ``fibotorial" numbers are obtained by replacing $n!$ with the product of the first $n$ Fibonacci numbers, and the ``fibinomial" coefficients are obtained by taking the formula for binomial coefficients and replacing the usual factorials with fibotorials. The polynomial analogue of this is what Sagan and Savage refer to as \emph{lucanomials}, which are defined as follows. For $n>1$, let
\[
\{ n\}! = \prod_{i=0}^{n-1} F_i(s,t),
\]
and set $\{0\}!=1$ by definition. Then with $0\leq k \leq n$, define  
\[
 \lbinom{n}{k} = \frac{\{n\}!}{\{k\}!\{n-k\}!}.
\]
For example, one can check that 
\[
 \lbinom{5}{3} = \frac{F_4(s,t)F_3(s,t)}{F_1(s,t)F_0(s,t)} = s^6 + 5s^4t + 7s^2t^2+2t^3.
\]

The fact that $\lbinom{n}{k}$ is a polynomial in $s$ and $t$ follows by induction from the recurrence
\[
 \lbinom{m+n}{m} = F_n(s,t)\lbinom{m+n-1}{m-1} + tF_{m-2}(s,t)\lbinom{m+n-1}{n-1}.
\]
The Sagan-Savage combinatorial model for the lucanomial coefficients takes a bit of explanation (this is the main content of their paper).

Before we give their combinatorial interpretation, recall an integer partition with $m$ parts is a sequence of nonnegative integers $\lambda=(\lambda_1,\lambda_2, \ldots,\lambda_m)$ with $\lambda_1\geq \lambda_2 \geq\dotsc \geq \lambda_m \geq 0$. We think of partitions visually in terms of an array of left-justified boxes with $\lambda_1$ boxes in the first row, $\lambda_2$ boxes in the second row, and so on, known as a \emph{Ferrers diagram}. We say that $\lambda$ is contained in an $m \times n$ rectangle if the largest part of $\lambda$ has size at most $n$, i.e., $\lambda_1 \leq n$. The \emph{complement} of $\lambda$ within the $m\times n$ rectangle is the partition $\lambda^*= (\lambda^*_1,\lambda_2^*,\dotsc,\lambda^*_n)$ where
$\lambda^*_i = m-\#\{j \mid \lambda_j  \geq n+1-i\}$.

\begin{figure}
\[
 \begin{tikzpicture}
  \draw (0,0) grid (4,5);
  \draw[line width=4,rounded corners, cap=round] (0,0)--(0,2)--(2,2)--(2,3)--(3,3)--(3,5)--(4,5);
 \end{tikzpicture}
 \hspace{2cm}
 \begin{tikzpicture}
  \draw (0,0) grid (4,5);
  \draw[line width=4,rounded corners, cap=round] (0,0)--(0,2)--(2,2)--(2,3)--(3,3)--(3,5)--(4,5);
  \foreach \x in {1,...,4}{
   \foreach \y in {1,...,5}{
    \draw (\x-.5,\y-.5) node[circle, inner sep=2, fill=black] {};
   }
  }
  \draw (.5,4.5)--(1.5,4.5);
  \draw (1.5,3.5)--(2.5,3.5);
  \draw (.5,.5)--(.5,1.5);
  \draw (1.5,.5)--(1.5,1.5);
  \draw (2.5,.5)--(2.5,1.5);
  \draw (3.5,.5)--(3.5,1.5);
  \draw (3.5,3.5)--(3.5,4.5);
 \end{tikzpicture}
\]
\caption{The partition $\lambda = (3,3,2,0,0)$ inside a $5 \times 4$ rectangle and a matching of weight $s^6t^7$.}\label{fig:complement}
\end{figure}
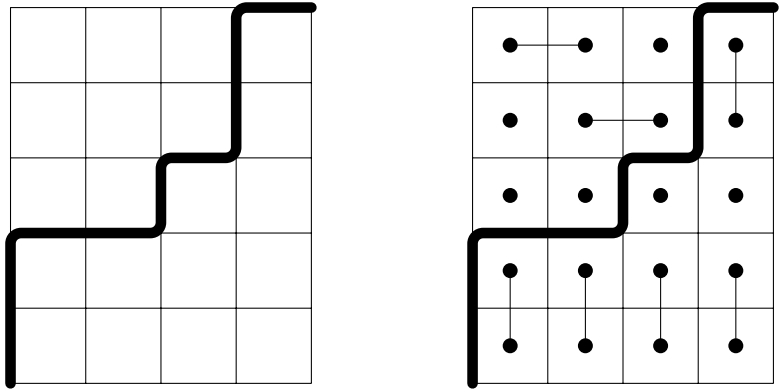

For example, in Figure \ref{fig:complement} we see the Ferrers diagram for the partition $\lambda=(3,3,2,0,0)$ in a $5 \times 4$ rectangle. Its complement is $\lambda^* = (5,3,2,2)$.

Now, for a given partition $\lambda = (\lambda_1,\ldots,\lambda_m)$, we let $\ml_{\lambda}$ denote the set of all tuples of matchings $T=(T_1,\ldots,T_m)$, where $T_i$ is a matching on the path graph $1\times \lambda_i$. We draw these matchings in the rows of the Ferrers diagram for $\lambda$. Let $\ml'_{\lambda}$ denote the set of all tuples of \emph{strict} matchings on the rows of $\lambda$, where by strict we mean the matching cannot begin with an isolated node. Notice there are no strict matchings of the path with one node and zero edges. 

For a tuple of matchings, $T$, we let $w(T)$ denote the weight $s^{m(T)}t^{d(T)}$, where $m(T)$ denotes the number of isolated nodes (monomers) in $T$ and $d(T)$ denotes the number of edges (dimers). In Figure \ref{fig:complement} we see an element $T$ of $\ml_{\lambda}\times\ml_{\lambda^*}'$, where we draw the strict matchings of $\lambda^*$ up the columns the complement of the Ferrers diagram of $\lambda$. For this example there are $6$ isolated nodes and $7$ edges, giving a weight of $w(T)=s^6t^7$.

One of the main results of Sagan and Savage's paper, \cite[Theorem 3]{Sagan-Savage}, is the fact that lucanomial coefficients count all such tilings according to weight. That is,
\begin{equation}\label{eq:lucweight}
 \lbinom{m+n}{m} = \sum_{\lambda \subseteq m\times n} \sum_{T \in \ml_{\lambda}\times \ml_{\lambda^*}'} w(T),
\end{equation}
where $\lambda \subseteq m\times n$ means the partition $\lambda$ fits inside an $m\times n$ rectangle. Using the matchings shown in Figure \ref{fig:lucastilings} we can verify the case $m=3$ and $n=2$ is given by $\lbinom{5}{3}=s^6 + 5s^4t + 7s^2t^2+2t^3$.

\begin{figure}
\[
\begin{array}{c | c}
\mbox{matching} & \mbox{weight} \\
\hline
& \\
\begin{tikzpicture}[scale=.5,baseline=.5cm]
  \draw (0,0) grid (2,3);
  \draw[line width=2,rounded corners, cap=round] (0,0)--(0,3)--(2,3);
  \foreach \x in {1,...,2}{
   \foreach \y in {1,...,3}{
    \draw (\x-.5,\y-.5) node[circle, inner sep=2, fill=black] {};
   }
  }
  \draw (.5,.5)--(.5,1.5);
  \draw (1.5,.5)--(1.5,1.5);
\end{tikzpicture}
&
s^2t^2 \\
& \\
\begin{tikzpicture}[scale=.5,baseline=.5cm]
  \draw (0,0) grid (2,3);
  \draw[line width=2,rounded corners, cap=round] (0,0)--(0,2)--(1,2)--(1,3)--(2,3);
  \foreach \x in {1,...,2}{
   \foreach \y in {1,...,3}{
    \draw (\x-.5,\y-.5) node[circle, inner sep=2, fill=black] {};
   }
  }
  \draw (.5,.5)--(.5,1.5);
  \draw (1.5,.5)--(1.5,1.5);
\end{tikzpicture}
&
s^2t^2 \\
& \\
\begin{tikzpicture}[scale=.5,baseline=.5cm]
  \draw (0,0) grid (2,3);
  \draw[line width=2,rounded corners, cap=round] (0,0)--(0,2)--(2,2)--(2,3);
  \foreach \x in {1,...,2}{
   \foreach \y in {1,...,3}{
    \draw (\x-.5,\y-.5) node[circle, inner sep=2, fill=black] {};
   }
  }
  \draw (.5,.5)--(.5,1.5);
  \draw (1.5,.5)--(1.5,1.5);
\end{tikzpicture}
&
s^2t^2 \\
& \\
\begin{tikzpicture}[scale=.5,baseline=.5cm]
  \draw (0,0) grid (2,3);
  \draw[line width=2,rounded corners, cap=round] (0,0)--(0,2)--(2,2)--(2,3);
  \foreach \x in {1,...,2}{
   \foreach \y in {1,...,3}{
    \draw (\x-.5,\y-.5) node[circle, inner sep=2, fill=black] {};
   }
  }
  \draw (.5,2.5)--(1.5,2.5);
  \draw (.5,.5)--(.5,1.5);
  \draw (1.5,.5)--(1.5,1.5);
\end{tikzpicture}
&
t^3 \\
& \\
\begin{tikzpicture}[scale=.5,baseline=.5cm]
  \draw (0,0) grid (2,3);
  \draw[line width=2,rounded corners, cap=round] (0,0)--(1,0)--(1,3)--(2,3);
  \foreach \x in {1,...,2}{
   \foreach \y in {1,...,3}{
    \draw (\x-.5,\y-.5) node[circle, inner sep=2, fill=black] {};
   }
  }
  \draw (1.5,.5)--(1.5,1.5);
\end{tikzpicture}
&
s^4t \\
& \\
\begin{tikzpicture}[scale=.5,baseline=.5cm]
  \draw (0,0) grid (2,3);
  \draw[line width=2,rounded corners, cap=round] (0,0)--(1,0)--(1,2)--(2,2)--(2,3);
  \foreach \x in {1,...,2}{
   \foreach \y in {1,...,3}{
    \draw (\x-.5,\y-.5) node[circle, inner sep=2, fill=black] {};
   }
  }
  \draw (1.5,.5)--(1.5,1.5);
\end{tikzpicture}
&
s^4t \\
& \\
\begin{tikzpicture}[scale=.5,baseline=.5cm]
  \draw (0,0) grid (2,3);
  \draw[line width=2,rounded corners, cap=round] (0,0)--(1,0)--(1,2)--(2,2)--(2,3);
  \foreach \x in {1,...,2}{
   \foreach \y in {1,...,3}{
    \draw (\x-.5,\y-.5) node[circle, inner sep=2, fill=black] {};
   }
  }
  \draw (.5,2.5)--(1.5,2.5);
  \draw (1.5,.5)--(1.5,1.5);
\end{tikzpicture}
&
s^2t^2 \\
& \\
\end{array}
\hspace{1cm}
\begin{array}{c | c}
\mbox{matching} & \mbox{weight} \\
\hline
& \\
\begin{tikzpicture}[scale=.5,baseline=.5cm]
  \draw (0,0) grid (2,3);
  \draw[line width=2,rounded corners, cap=round] (0,0)--(2,0)--(2,3);
  \foreach \x in {1,...,2}{
   \foreach \y in {1,...,3}{
    \draw (\x-.5,\y-.5) node[circle, inner sep=2, fill=black] {};
   }
  }
\end{tikzpicture}
&
s^6 \\
& \\
\begin{tikzpicture}[scale=.5,baseline=.5cm]
  \draw (0,0) grid (2,3);
  \draw[line width=2,rounded corners, cap=round] (0,0)--(2,0)--(2,3);
  \foreach \x in {1,...,2}{
   \foreach \y in {1,...,3}{
    \draw (\x-.5,\y-.5) node[circle, inner sep=2, fill=black] {};
   }
  }
  \draw (.5,2.5)--(1.5,2.5);
\end{tikzpicture}
&
s^4t \\
& \\
\begin{tikzpicture}[scale=.5,baseline=.5cm]
  \draw (0,0) grid (2,3);
  \draw[line width=2,rounded corners, cap=round] (0,0)--(2,0)--(2,3);
  \foreach \x in {1,...,2}{
   \foreach \y in {1,...,3}{
    \draw (\x-.5,\y-.5) node[circle, inner sep=2, fill=black] {};
   }
  }
  \draw (.5,1.5)--(1.5,1.5);
\end{tikzpicture}
&
s^4t \\
& \\
\begin{tikzpicture}[scale=.5,baseline=.5cm]
  \draw (0,0) grid (2,3);
  \draw[line width=2,rounded corners, cap=round] (0,0)--(2,0)--(2,3);
  \foreach \x in {1,...,2}{
   \foreach \y in {1,...,3}{
    \draw (\x-.5,\y-.5) node[circle, inner sep=2, fill=black] {};
   }
  }
  \draw (.5,.5)--(1.5,.5);
\end{tikzpicture}
&
s^4t \\
& \\
\begin{tikzpicture}[scale=.5,baseline=.5cm]
  \draw (0,0) grid (2,3);
  \draw[line width=2,rounded corners, cap=round] (0,0)--(2,0)--(2,3);
  \foreach \x in {1,...,2}{
   \foreach \y in {1,...,3}{
    \draw (\x-.5,\y-.5) node[circle, inner sep=2, fill=black] {};
   }
  }
  \draw (.5,2.5)--(1.5,2.5);
  \draw (.5,1.5)--(1.5,1.5);
\end{tikzpicture}
&
s^2t^2 \\
& \\
\begin{tikzpicture}[scale=.5,baseline=.5cm]
  \draw (0,0) grid (2,3);
  \draw[line width=2,rounded corners, cap=round] (0,0)--(2,0)--(2,3);
  \foreach \x in {1,...,2}{
   \foreach \y in {1,...,3}{
    \draw (\x-.5,\y-.5) node[circle, inner sep=2, fill=black] {};
   }
  }
  \draw (.5,2.5)--(1.5,2.5);
  \draw (.5,.5)--(1.5,.5);
\end{tikzpicture}
&
s^2t^2 \\
& \\
\begin{tikzpicture}[scale=.5,baseline=.5cm]
  \draw (0,0) grid (2,3);
  \draw[line width=2,rounded corners, cap=round] (0,0)--(2,0)--(2,3);
  \foreach \x in {1,...,2}{
   \foreach \y in {1,...,3}{
    \draw (\x-.5,\y-.5) node[circle, inner sep=2, fill=black] {};
   }
  }
  \draw (.5,.5)--(1.5,.5);
  \draw (.5,1.5)--(1.5,1.5);
\end{tikzpicture}
&
s^2t^2 \\
& \\
\begin{tikzpicture}[scale=.5,baseline=.5cm]
  \draw (0,0) grid (2,3);
  \draw[line width=2,rounded corners, cap=round] (0,0)--(2,0)--(2,3);
  \foreach \x in {1,...,2}{
   \foreach \y in {1,...,3}{
    \draw (\x-.5,\y-.5) node[circle, inner sep=2, fill=black] {};
   }
  }
  \draw (.5,2.5)--(1.5,2.5);
  \draw (.5,1.5)--(1.5,1.5);
  \draw (.5,.5)--(1.5,.5);
\end{tikzpicture}
& t^3 \\
& \\
\end{array}
\]
\caption{The fifteen tuples of matchings counted by $\lbinom{5}{3}$.}\label{fig:lucastilings}
\end{figure}

To phrase \eqref{eq:lucweight} a bit differently, let $\ml_{\lambda,i}$ denote the set of tuples of matchings in $\ml_{\lambda}\times \ml_{\lambda^*}'$ that have exactly $i$ edges, and let 
\[
 \ml(m+n,m,i) = \bigcup_{\lambda \subseteq m\times n} \ml_{\lambda,i}.
\]
Then, by analogy with Equation \eqref{eq:binomexpansion} we have
\begin{equation}\label{eq:lucasbintiling}
 \lbinom{m+n}{m}  = \sum_{i\geq 0} |\ml(m+n,m,i)|s^{mn-2i} t^i.
\end{equation}

Recalling the specialization $s=1+q$ and $t=-q$ gives $F_n(1+q,-q) = [n+1]$, we have
\[
 \{ n\}!_{s=1+q, t=-q} = \prod_{i=0}^{n-1} F_i(1+q,-q) = [n]!. 
\]
Thus,
\begin{equation}\label{eq:lbintoqbin}
 \lbinom{n}{k}_{s=1+q, t=-q} = \frac{[n]!}{[k]![n-k]!} =\qbinom{n}{k},
\end{equation}
which is known as the \emph{$q$-binomial coefficient}.

The $q$-binomial coefficients have a combinatorial interpretation given by counting lattice paths from $(0,0)$ to $(k,n-k)$ according to area beneath the path. From this it follows that $\qbinom{n}{k}$ is palindromic of palindromic degree $k(n-k)$. This claim can also be justified by observing 
\[
 [k]![n-k]!\qbinom{n}{k} = [n]!,
\]
and $[n]!$ has palindromic degree $1+2+\cdots + (n-1)=\binom{n}{2}$.

In any event, we can put \eqref{eq:lbintoqbin} together with \eqref{eq:lucasbintiling} to obtain
\[
 \qbinom{n}{k} = \lbinom{n}{k}_{s=1+q, t=-q} = \sum_{i\geq 0} (-1)^i|\ml(n,k,i)|q^i(1+q)^{k(n-k)-2i}.
\] 
Therefore the $\gamma$-polynomial of $\qbinom{n}{k}$ is $\lbinom{n}{k}$ evaluated at $s=1$ and $t=-z$.

\begin{prop}\label{prop:gam-qbin}
The $\gamma$-polynomial for the $q$-binomial coefficients is the corresponding lucanomial coefficient evaluated at $s=1, t=-z$, i.e., $\qbinom{n}{k}$ has $\gamma$-polynomial
\[
 \lbinom{n}{k}_{s=1,t=-z}.
\]
\end{prop}

As a combinatorial consequence, we have the following.

\begin{cor}\label{cor:lbinomial}
If $(\gamma_0, \ldots,\gamma_{\lfloor d/2 \rfloor})$ is the $\gamma$-vector of $\qbinom{n}{k}$, then
\[
 \gamma_j = (-1)^j|\ml(n,k,j)|.
\]
\end{cor}

Define the set
\[
 \ml_{n,k}=\bigcup_{i \geq 0} \ml(n,k,i),
\]
i.e., all possible tuples of matchings that come from partitions that fit in a $k\times (n-k)$ rectangle. Call such a tuple a \emph{lucanomial matching}. Then Corollary \ref{cor:lbinomial} means that, up to sign, the $j$th entry of the $\gamma$-vector for $\qbinom{n}{k}$ counts lucanomial matchings with $j$ edges.

\section{Combinatorial expressions for $g$-vectors}
\label{sec:gamma-to-g}

To this point we have shown that for each of three families of polynomials, the gamma vectors are given by a signed count of certain matchings according to the number of edges. For $[n]!$ we count ``fibotorial matchings" in $\mt_n$, for $\prod_{j=1}^n (1+q^j)$ we count ``lucatorial matchings" in $\mt'_n$, and for $\qbinom{n}{k}$ we count ``lucanomial matchings" in $\ml_{n,k}$.

Returning to the unimodality paradigm outlined in Section \ref{sec:indirect}, we want to interpret the entries of the $g$-vector as a linear combination of entries in the corresponding $\gamma$-vector. Applying Equation \eqref{eq:gamtog} in each circumstance yields a combinatorial alternating sum. Generically, we have
\[
g_i = \sum_{j\geq 0} \gamma_j B_d(i,j),
\]
where $d$ is the palindromic degree, and we recall that $B_d(i,j) =|\bp_d(i,j)|$ counts the number of ballot paths of length of $d-2j$ with $i-j$ North steps. We summarize each of these combinatorial formulas in the following theorem.

\begin{thm}\label{thm:main}
The entries of the $g$-vector have the following expressions.
\begin{enumerate}
\item  When $g$ is the $g$-vector of $[n]!$,
\[g_i = \sum_{j\geq 0} (-1)^j|\mt(n,j)\times \bp_d(i,j)|,\]
where $d=\binom{n}{2}$.
\item When $g$ is the $g$-vector of $\prod_{k=1}^n (1+q^k)$,
\[g_i = \sum_{j\geq 0} (-1)^j|\mt'(n,j)\times \bp_d(i,j)|,\]
where $d=\binom{n+1}{2}$.
\item When $g$ is the $g$-vector of $\qbinom{n}{k}$,
\[g_i = \sum_{j\geq 0} (-1)^j|\ml(n,k,j)\times \bp_d(i,j)|,\]
where $d=k(n-k)$.
\end{enumerate}
\end{thm}

To prove unimodality, in the first case say, the goal is to find a sign-reversing involution on the set 
\begin{equation}\label{eq:X}
 X= \bigcup_{j \geq 0} \mt(n,j)\times \bp_d(i,j),
\end{equation}
where $i$ is fixed, $d=\binom{n}{2}$, and the sign of a pair $(T,p)$ in $X$ is given by the parity of the number of edges in $T$.

We will demonstrate such an involution for the case of $[n]!$ in Section \ref{sec:n!inv}, though at present we cannot execute this plan in the other cases. We discuss our partial progress in Section \ref{sec:hardcases}

\subsection{A new proof of unimodality for $q$-factorials}\label{sec:n!inv}

Let $X=X_{d,i}$ denote the set of all pairs of the form $(T,p)$ as as described in \eqref{eq:X}. To define our sign-reversing involution on $X$, we will first create a set $Y$ that is in bijection with $X$ to make our involution easier to understand. 

Let $Y= Y_{d,i}$ denote the set of all \emph{decorated ballot paths} of length $d$ with $i$ North steps. These are ballot paths in the usual sense, except that some of the vertices and edges can come in more than one style. First, with $d=1 + 2 +\cdots + (n-1)$, we will draw the vertices $1$, $3$, $6$, \ldots, $(1+2+\cdots +i)$, \ldots in a different color and refer to these vertices as \emph{anchors}. For example, Figure \ref{fig:decpath} shows a path of length $d = 15$. The anchors are in white.

Every decorated path of length $d$ has the same anchor nodes. The interesting decorations come from ``valleys" in the path that are not anchors, i.e., consecutive steps of the form $EN$. We call such a valley an \emph{active valley}. Each such valley can be decorated or not. There are four active valleys shown in Figure \ref{fig:decpath}, and the first and third of these are decorated. In the picture these are indicated with a dashed line.

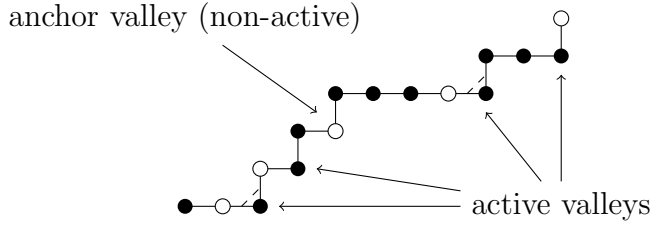
\begin{figure}
\[
\begin{tikzpicture}[scale=.5]
\draw (0,0) node[circle, inner sep=2, fill=black] {} -- (1,0) node[circle, inner sep=2, draw=black, fill=white] {} -- (2,0) node[circle, inner sep=2, fill=black] {} -- (2,1) node[circle, inner sep=2, draw=black, fill=white] {} -- (3,1) node[circle, inner sep=2, fill=black] {}-- (3,2) node[circle, inner sep=2, fill=black] {}-- (4,2) node[circle, inner sep=2, draw=black, fill=white] {} -- (4,3) node[circle, inner sep=2, fill=black] {} -- (5,3) node[circle, inner sep=2, fill=black] {} -- (6,3) node[circle, inner sep=2, fill=black] {} -- (7,3) node[circle, inner sep=2, draw=black, fill=white] {} -- (8,3) node[circle, inner sep=2, fill=black] {} -- (8,4) node[circle, inner sep=2, fill=black] {} -- (9,4) node[circle, inner sep=2, fill=black] {} -- (10,4) node[circle, inner sep=2, fill=black] {}-- (10,5) node[circle, inner sep=2, draw=black, fill=white] {};
\draw (10,0) node (a) {active valleys};
\draw[->] (a)--(2.5,0);
\draw[->] (a)--(3.5,1);
\draw[->] (a)--(8,2.5);
\draw[->] (a)--(10,3.5);
\draw (0,5) node (b) {anchor valley (non-active)};
\draw[->] (b)--(3.5,2.5);
\draw[dashed] (1.5,0)--(2,.5);
\draw[dashed] (7.5,3)--(8,3.5);
\end{tikzpicture}
\]
\caption{A decorated path of length $d=15$.}\label{fig:decpath}
\end{figure}

Notice that while there is a valley after the sixth step in the path, it is not active since it contains an anchor node. When writing decorated paths as words, we put vertical bars in the anchor positions. The path above is written as follows:
\[
 \mathbf{p} = E|en|ENE|NEEE|enEEN.
\]
The decorated valleys are indicated with lowercase letters.

The correspondence between a decorated path like this one and a pair is straightforward. Suppose $\mathbf{p}$ is a decorated ballot path of length $d$ with $i$ North steps and $j$ decorated valleys. We want to associate $\mathbf{p}$ with a pair $(T,p)$, where $T$ is a fibotorial matching with $j$ edges and $p$ is a ballot path with $d-2j$ steps and $i-j$ North steps.

This pair is nearly obvious in $\mathbf{p}$. The edges of the matching $T$ correspond precisely to adjacent lowercase valleys, $en$, that don't have a bar in them. To get the ballot path, we delete all lowercase letters and bars from $\mathbf{p}$. As there are $j$ lowercase ``$en$" pairs, this leaves a ballot path with $d-2j$ steps, $i-j$ of which are North steps.

For example, with $\mathbf{p} = E|en|ENE|NEEE|enEEN$ as in Figure \ref{fig:decpath}, we get
\[
 T = 
( \begin{tikzpicture}
\foreach \x in {0,...,0}{
\draw (\x,0) node[circle, inner sep=2, fill=black] {};
}
\end{tikzpicture},
\begin{tikzpicture}[scale=.5]
\foreach \x in {0,...,1}{
\draw (\x,0) node[circle, inner sep=2, fill=black] {};
}
\draw (0,0)--(1,0);
\end{tikzpicture},
\begin{tikzpicture}[scale=.5]
\foreach \x in {0,...,2}{
\draw (\x,0) node[circle, inner sep=2, fill=black] {};
}
\end{tikzpicture},
\begin{tikzpicture}[scale=.5]
\foreach \x in {0,...,3}{
\draw (\x,0) node[circle, inner sep=2, fill=black] {};
}
\end{tikzpicture},
\begin{tikzpicture}[scale=.5]
\foreach \x in {0,...,4}{
\draw (\x,0) node[circle, inner sep=2, fill=black] {};
}
\draw (0,0)--(1,0);
\end{tikzpicture}
),
\]
and
\[ 
p=E\,ENE\,NEEE\,EEN.
\]

Constructing a decorated path $\mathbf{p}$ from a pair $(T,p)$ is equally straightforward. We think of the letters of $p$ as corresponding to unmatched nodes in $T$. The matched nodes of $T$ correspond to lowercase $en$ pairs, which we insert as the decorated valleys in $\mathbf{p}$. (Notice that inserting an $en$ pair into a ballot path results in another ballot path, i.e., one which does not cross the line $y=x$.) For example, if 
\[
 T = 
( \begin{tikzpicture}
\foreach \x in {0,...,0}{
\draw (\x,0) node[circle, inner sep=2, fill=black] {};
}
\end{tikzpicture},
\begin{tikzpicture}[scale=.5]
\foreach \x in {0,...,1}{
\draw (\x,0) node[circle, inner sep=2, fill=black] {};
}
\end{tikzpicture},
\begin{tikzpicture}[scale=.5]
\foreach \x in {0,...,2}{
\draw (\x,0) node[circle, inner sep=2, fill=black] {};
\draw (1,0)--(2,0);
}
\end{tikzpicture},
\begin{tikzpicture}[scale=.5]
\foreach \x in {0,...,3}{
\draw (\x,0) node[circle, inner sep=2, fill=black] {};
}
\end{tikzpicture},
\begin{tikzpicture}[scale=.5]
\foreach \x in {0,...,4}{
\draw (\x,0) node[circle, inner sep=2, fill=black] {};
}
\draw (0,0)--(1,0);
\draw (2,0)--(3,0);
\end{tikzpicture}
), 
\]
and
\[
p=E\, EE\, N \,EENN\, E,
\]
then 
\[
\mathbf{p} = E|EE|Nen|EENN|enenE.
\]

Let us summarize this correspondence in the following proposition.

\begin{prop}
The sets $X_{d,i}$ and $Y_{d,i}$ are in bijection as described above. If 
\[
 \mathbf{p} \leftrightarrow (T,p),
\]
the number of decorated valleys in $\mathbf{p}$ equals the number of edges in $T$.\end{prop}

Given this proposition, we construct an involution $\iota : Y \to Y$ that matches decorated paths with an odd number of decorated valleys with those having an even number of decorated valleys.

The involution is simple. If $\mathbf{p}$ is a decorated path with at least one active valley, then $\iota$ ``toggles" the first active valley from upper to lower case or vice-versa, while keeping the rest of the path fixed. For example, if
\[
\mathbf{p} = E|EE|Nen|EENN|enenE,
\]
then
\[
\iota(\mathbf{p}) = E|EE|NEN|EENN|enenE.
\]

If $\mathbf{p}$ has no active valleys (e.g., if $\mathbf{p}$ is the path consisting of all East steps), then $\iota(\mathbf{p}) = \mathbf{p}$.

In either case, it is clear that $\iota(\iota(\mathbf{p}))=\mathbf{p}$ for all $\mathbf{p}$ in $Y$, so $\iota$ is an involution.

The $\iota$ map changes the parity of the number of decorated valleys, so it does indeed change sign, and all the fixed points have zero active valleys (and hence positive sign). As a corollary of this involution we can characterize the $g$-vector of $[n]!$.

\begin{cor}
The entries of the $g$-vector for $[n]!$ are given by
\[
 g_i = |\{ \mathbf{p} \in Y_{d,i} : \mathbf{p} \mbox{ has no active valleys }\}|, 
\]
where $d = \binom{n}{2}$.
\end{cor}

We know of no simpler (manifestly positive) description for the entries of this $g$-vector. As these numbers are not available in the Online Encyclopedia of Integer Sequences (OEIS) \cite{oeis}, we provide some more detail here.

Let $g_{n,i}$ denote the $i$th entry for the $g$-vector of $[n]!$, where $i$ ranges from $0$ to $\lfloor d/2 \rfloor = \lfloor \binom{n}{2}/2 \rfloor$ and $g_{n,i}=0$ otherwise. We have these numbers, for small $n$, in Table \ref{tab:gni}. 

\begin{table}
\[
 \begin{array}{c|rrrrrrrrrrrrrrrrrr}
  n \backslash k & 0 & 1 & 2 & 3 & 4 & 5 & 6 & 7 & 8 & 9 & 10 & 11 & 12 & 13 & 14  \\
  \hline
  1 & 1 \\
  2 & 1 \\
  3 & 1 & 1 \\
  4 & 1 & 2 & 2 & 1 \\
  5 & 1 & 3 & 5 & 6 & 5 & 2 \\
  6 & 1 & 4 & 9 & 15 & 20 & 22 & 19 & 11\\
  7 & 1 & 5 & 14 & 29 & 49 & 71 & 90 & 100 & 96 & 76 & 42\\
  8 & 1 & 6 & 20 & 49 & 98 & 169 & 259 & 359 & 454 & 525 & 553 & 524 & 433 & 286 & 100
 \end{array}
\]
\caption{The entries of the $g$-vector for $[n]!$.}\label{tab:gni}
\end{table}

From our lattice path description, we can show these numbers are given by the following linear recurrence. 

\begin{prop}\label{prop:rec}
For any $n \geq 2$, we have 
\[
 g_{n,i} = \begin{cases} \displaystyle\sum_{j = i-(n-1)}^i g_{n-1,j} & \mbox{if } i \leq \lceil \binom{n-1}{2}/2 \rceil,\\
  \displaystyle\sum_{j=i-(n-1)}^{\binom{n}{2} -i - (n-1)} g_{n-1,j} & \mbox{if } i > \lceil \binom{n-1}{2}/2 \rceil.
  \end{cases} 
\]
\end{prop}

\begin{proof}
A valley-less path is a path that has no $EN$ in it. There are clearly $n$ valley-less paths with $n-1$ steps: there can be no North steps, $EE\cdots E$, one North step, $NEE\cdots E$, two North steps, $NNE\cdots E$, and so on, up to the path with all North steps, $NN\cdots N$.

Each decorated ballot path with no active valleys consists of the concatenation of ordinary valley-less paths, of lengths $1$, $2$, $3$, and so on. However, the ballot condition means that not every concatenation of valley-less paths is a decorated ballot path with no active valleys. 

Suppose $p = p_1 | p_2|\cdots |p_{n-2}|p_{n-1}$ is a decorated ballot path with no active valleys, so that $p_k$ is a valley-less path of length $k$. Then the path $p' = p_1 | p_2|\cdots |p_{n-2}$ is also a decorated ballot path with no active valleys. Moreover, if $p$ has $i$ North steps and $p'$ has $j$ North steps, then $p_{n-1}$ is the unique valley-less path of length $n-1$ with $i-j$ North steps.

Thus if we group the elements of $Y_{\binom{n}{2},i}$ according to the number of North steps in their final valley-less path, we get
\[
 g_{n,i} \leq g_{n-1,i} + g_{n-1,i-1} + g_{n-1,i-2} + \cdots + g_{n-1,i-(n-1)},
\]
since this final path can have anywhere from $0$ to $n-1$ North steps.

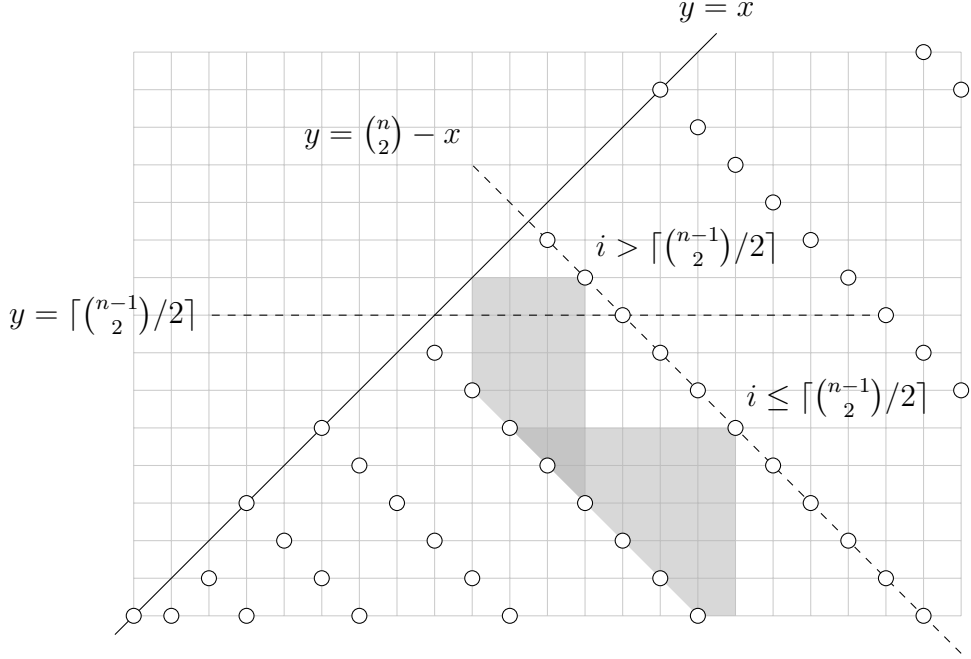
\begin{figure}
\[
 \begin{tikzpicture}[scale=.5]
 \draw[color=white!80!black] (0,0) grid (22,15);
 \draw[draw=none,fill=white!70!black, opacity=.5] (12,9)--(12,3)--(9,6)--(9,9)--(12,9)--cycle;
 \draw[draw=none,fill=white!70!black, opacity=.5] (16,5)--(16,0)--(15,0)--(10,5)--(16,5)--cycle;
 \draw (-.5,-.5)--(15.5,15.5) node[above] {$y=x$};
 \draw[dashed] (22,-1)--(9,12) node[above left] {$y=\binom{n}{2}-x$};
 \draw[dashed] (20,8)--(2,8) node[left] {$y=\lceil \binom{n-1}{2}/2 \rceil$};
 \draw (22,14) node[circle,draw=black,fill=white,inner sep=2] {};
 \draw (21,15) node[circle,draw=black,fill=white,inner sep=2] {};
 \foreach \x in {6,...,14}{
  \draw (28-\x,\x) node[circle,draw=black,fill=white,inner sep=2] {};
 }
 \foreach \x in {0,...,10}{
  \draw (21-\x,\x) node[circle,draw=black,fill=white,inner sep=2] {};
 }
 \foreach \x in {0,...,7}{
  \draw (15-\x,\x) node[circle,draw=black,fill=white,inner sep=2] {};
 }
 \foreach \x in {0,...,5}{
  \draw (10-\x,\x) node[circle,draw=black,fill=white,inner sep=2] {};
 }
 \foreach \x in {0,...,3}{
  \draw (6-\x,\x) node[circle,draw=black,fill=white,inner sep=2] {};
 }
 \foreach \x in {0,...,1}{
  \draw (3-\x,\x) node[circle,draw=black,fill=white,inner sep=2] {};
 }
 \draw (1,0) node[circle,draw=black,fill=white,inner sep=2] {};
 \draw (0,0) node[circle,draw=black,fill=white,inner sep=2] {};
 \draw (12,9) node[above right] {$i>\lceil \binom{n-1}{2}/2 \rceil$};
  \draw (16,5) node[above right] {$i\leq \lceil \binom{n-1}{2}/2 \rceil$};
 \end{tikzpicture}
\]
\caption{The boundary conditions for the recurrence given in Proposition \ref{prop:rec}. Possible locations for anchor nodes are indicated in white.}\label{fig:boundary}
\end{figure}

Fix $j$ between $i$ and $i-(n-1)$. 

For some choices of $j$, every ballot path $p'$ of length $\binom{n-1}{2}$ with $j$ North steps can be uniquely extended to a ballot path of length $\binom{n}{2}$ with the valley-less path consisting of $i-j$ North steps. In this case the inequality above is an equality. This can be done precisely when $i \leq \lceil \binom{n-1}{2} \rceil$.

However, if $j$ and $i$ are too close to $\lfloor \binom{n}{2}/2\rfloor$, appending this valley-less path will take the lattice path above the line $y=x$, violating the ballot condition. By careful examination, we find this happens when $i > \lceil \binom{n-1}{2}/2 \rceil$ and $j > \binom{n}{2}-i-(n-1)$. Hence we arrive at the cases stated in the proposition. See the illustration in Figure \ref{fig:boundary}.
\end{proof}

\subsection{The harder cases}\label{sec:hardcases}

We would like to report that we can replicate the approach of Section \ref{sec:n!inv} for the polynomials $\prod_{k=1}^n(1+q^k)$ and $\qbinom{n}{k}$. The combinatorial setup is there in Theorem \ref{thm:main}. All that we lack is a clever sign-reversing involution.

For $\prod_{k=1}^n(1+q^k)$, the difficulty seems to be that there is no canoncial linear ordering on the edges in a cyclic matching, and hence ballot paths and lucatorial matchings are not as obviously compatible as ballot paths and fibotorial matchings.

With $\qbinom{n}{k}$, the edges within each row of a given Ferrers diagram $\lambda$ (and within each column of $\lambda^*$) are linearly ordered. However, the tuple of matchings of the rows does not have a canonical ordering. Moreover, as our lucanomial matchings range over different partitions $\lambda$, the number and size of the rows varies. Thus lucanomial matchings prove tricky to relate to ballot paths as well. A sign-reversing involution using this model will require some subtlety.

\end{document}